\documentclass[12pt,reqno]{amsart}

\newcommand{\alp}{\alpha}
\newcommand{\bet}{\beta}
\newcommand{\gam}{\gamma}
\newcommand{\eps}{\varepsilon}
\newcommand{\kap}{\kappa}
\newcommand{\lam}{\lambda}
\newcommand{\sig}{\sigma}

\newcommand{\R}{{\mathbb R}}
\newcommand{\Z}{{\mathbb Z}}

\newcommand{\cS}{{\mathcal S}}

\newcommand{\hA}{\widehat{A}}

\newcommand{\oA}{\overline{A}}
\newcommand{\oS}{\overline{S}}

\newcommand{\od}{\overline{d}}
\newcommand{\oz}{\overline{z}}

\newcommand{\lfl}{\left\lfloor}
\newcommand{\rfl}{\right\rfloor}
\newcommand{\lcl}{\left\lceil}
\newcommand{\rcl}{\right\rceil}
\newcommand{\lfr}{\left\{}
\newcommand{\rfr}{\right\}}

\newcommand{\stm}{\setminus}
\newcommand{\seq}{\subseteq}
\DeclareMathOperator{\mes}{mes}

\newcommand{\mmod}[1]{\!\!\pmod{#1}}

\newcommand{\longc}{,\ldots,}

\newcommand{\longp}{+\ldots+}

\newcounter{mypar}
\newcommand{\mypar}%
  {\medskip\stepcounter{mypar}\emph{Step }\arabic{mypar}.\ \:}

\theoremstyle{plain}
\newtheorem{lemma}{Lemma}
\newtheorem{theorem}{Theorem}
\newtheorem*{maintheorem}{Main Theorem}
\newtheorem{corollary}{Corollary}

\theoremstyle{definition}
\newtheorem{example}{Example}

\newcommand{\refl}[1]{~\ref{l:#1}}
\newcommand{\reft}[1]{~\ref{t:#1}}
\newcommand{\refc}[1]{~\ref{c:#1}}
\newcommand{\refx}[1]{~\ref{x:#1}}
\newcommand{\refs}[1]{~\ref{s:#1}}
\newcommand{\refb}[1]{~\cite{b:#1}}
\newcommand{\refe}[1]{~\eqref{e:#1}}

\title{The continuous postage stamp problem}
\author{Vsevolod F. Lev}
\address{Department of Mathematics, The University of Haifa at Oranim,
  Tivon 36006, Israel.}
\email{seva@math.haifa.ac.il}

\begin{document}
\baselineskip 16pt

\begin{abstract}
For a real set $A$ consider the semigroup $\cS(A)$, additively generated by
$A$; that is, the set of all real numbers representable as a (finite) sum of
elements of $A$. If $A\seq(0,1)$ is open and non-empty, then $\cS(A)$ is
easily seen to contain all sufficiently large real numbers, and we let
 $G(A):=\sup\{u\in\R\colon u\notin\cS(A)\}$. Thus $G(A)$ is the smallest
number with the property that any $u>G(A)$ is representable as indicated
above.

We show that if the measure of $A$ is large, then $G(A)$ is small; more
precisely, writing for brevity $\alp:=\mes A$ we have
  $$ G(A) \le \begin{cases}
                (1-\alp)\lfl 1/\alp\rfl
                   \quad&\text{if \, $0<\alp\le 0.1$}, \\
                (1-\alp+\alp\{1/\alp\})\lfl 1/\alp\rfl
                   \quad&\text{if \, $0.1 \le \alp \le 0.5$}, \\
                2(1-\alp)
                   \quad&\text{if \, $0.5 \le \alp \le 1$}. \\
              \end{cases} $$
Indeed, the first and the last of these three estimates are the best
possible, attained for $A=(1-\alp,1)$ and $A=(1-\alp,1)\stm\{2(1-\alp)\}$,
respectively; the second is close to the best possible and can be improved by
$\alp\{1/\alp\}\lfl 1/\alp\rfl\le\{1/\alp\}$ at most.

The problem studied is a continuous analogue of the linear Diophantine
problem of Frobenius (in its extremal settings due to Erd\H{o}s and Graham),
also known as the ``postage stamp problem'' or the ``coin exchange problem''.
\end{abstract}

\maketitle

\section{Background: the discrete postage stamp problem}

Let $A$ be a non-empty set of positive integers such that $\gcd(A)=1$. It is
not difficult to see that all sufficiently large integers can be represented
as a sum of elements of $A$. The problem of determining the largest integer
$G(A)$ which does not have such a representation is
known as the ``linear Diophantine problem of Frobenius''. Though this problem
is computational by its nature, there are numerous papers concentrating on
\emph{estimates} of $G(A)$.

Erd\H os and Graham posed in \refb{eg} the following extremal version of the
Frobenius problem: given two positive integers $l\ge n$, estimate
\begin{equation}\label{e:disc}
  \max\{G(A)\colon A\seq[1,l],\, |A|=n,\, \gcd(A)=1\}.
\end{equation}
The ideology here is that if $A$ is dense, then $G(A)$ must be small. The
basic result on the problem of Frobenius-Erd\H os-Graham was obtained by
Dixmier in \refb{d}. (Several years later it was independently established by
the present author as a corollary of a general theorem on set addition; see
\refb{l2}.) The exact value of \refe{disc} is known for some particular
families of pairs $(n,l)$ only, and also for $l\le 3n-2$; see \refb{l1}.

\section{Preliminaries: examples and summary of results}\label{s:prelim}

Motivated by their work on primality testing, Lenstra and Pomerance \refb{lp}
stated recently an analogue of the Frobenius-Erd\H os-Graham problem for
bounded sets $A$ of positive \emph{real} numbers. The condition $\gcd(A)=1$
becomes then irrelevant, and scaling allows one to assume $A\seq(0,1)$; thus,
the parameter $l$ is gone. Furthermore, following \refb{lp} we assume that
$A$ is open; this ensures measurability and forces $A$ to be a finite or
countable union of open intervals. Let $\cS(A)$ denote the set of all
numbers, representable as a finite sum of elements of $A$; thus
  $$ \textstyle \cS(A) = \bigcup_{h=1}^\infty\; hA, $$
where $hA$ is the set of all sums of exactly $h$ elements of $A$.

Is it true, by the analogy with the integer case, that $\cS(A)$ contains all
sufficiently large numbers? Set
  $$ G(A):=\sup\{u\in\R \colon u\notin\cS(A)\}; $$
the question therefore is whether $G(A)$ is finite for any open non-empty
subset $A\seq(0,1)$. To proceed, we work out a simple, yet important,
example.

\begin{example}\label{x:first}
Suppose that $A$ consists of one single interval: $A=(\bet,\gam)$, where
$\gam>\bet\ge 0$. Then for $h\ge 1$ we have $hA=(h\bet,h\gam)$, hence $hA$
and $(h+1)A$ intersect if and only if $h\gam>(h+1)\bet$, or equivalently
$h>\bet/(\gam-\bet)$; that is,
  $$ h \ge \lfl \frac{\bet}{\gam-\bet} \rfl + 1
                                      = \lfl \frac{\gam}{\gam-\bet}\rfl. $$
It follows that $G(A)=\lfl \gam/(\gam-\bet)\rfl \bet$. In particular, if
$A=(1-\alp,1)$ with $\alp\in(0,1]$ (so that $\alp=\mes A$), then
$G(A)=(1-\alp)\lfl1/\alp\rfl$.
\end{example}

Back to the general case, suppose that $A\seq(0,1)$ is open and non-empty.
Then $A$ contains an open interval; say, $(\bet,\gam)\seq A$. In view of
Example \refx{first} every $u>\lfl \gam/(\gam-\bet)\rfl \bet$ is
representable as a sum of elements of $(\bet,\gam)$, hence as a sum of
elements of $A$; this shows that $G(A)$ is finite and indeed that
 $G(A)\le\lfl \gam/(\gam-\bet)\rfl \bet$.

The exact problem raised in \refb{lp} is to determine, for a fixed real
number $s>0$, how large the logarithmic measure $\int_A t^{-1}\,dt$ of an
open subset $A\seq(0,1)$ can be, given that $s\notin\cS(A)$. This is
essentially equivalent to the following question: given the logarithmic
measure of $A$, how large can $G(A)$ be? It is this question that is
investigated in our paper, except that, following the lines of the classical
Frobenius-Erd\H os-Graham problem, we will be concerned with the regular
Lebesgue measure $\mes A$ rather than the logarithmic measure.\!\!
\footnote{
For a complete solution of the original problem of Lenstra and Pomerance see
the paper of Bleichenbacher \refb{b}. The method of Bleichenbacher is based
on linear programming and can be applied to handle the Lebesgue measure, too;
however in this case it yields somewhat weaker results.}

We now present two constructions of open sets $A\seq(0,1)$ such that, letting
$\alp=\mes A$, we have $G(A)>(1-\alp)\lfl1/\alp\rfl$ (compare with Example
\refx{first}).
\begin{example}\label{x:second}
Suppose that $\alp\in(1/2,1]$ and set $A:=(1-\alp,1)\stm\{2(1-\alp)\}$. Then
for $h\ge 2$ we have $hA=(h(1-\alp),h)\supseteq[h-1,h)$, whence
  $$ \textstyle \cS(A)\supseteq\bigcup_{h=2}^\infty[h-1,h)=[1,\infty) $$
and therefore $G(A)=2(1-\alp)$.
\end{example}

Our second construction is more elaborate.
\begin{example}\label{x:third}
Suppose that $\alp\in(1/3,1/2)$. Write
  $$ k := \lcl \frac1{1-2\alp}\rcl-2,\quad x := 1-\frac1{k+2},%
          \quad \text{and}\quad t:= 2\Big(1-\frac1{k+3}\Big)(1-\alp), $$
so that $k\ge 2$. One checks easily that
\begin{gather*}
  \frac12\,\Big( 1-\frac1{k+1} \Big) < \alp
                              \le \frac12\,\Big( 1-\frac1{k+2} \Big) \\
\intertext{and}
  1-\frac1k < x < \frac1t \le 1,
\end{gather*}
and we let
  $$ A := \Big( \frac{tx}k\,,\frac tk \Big)\cup
          \Big( \frac{2tx}k\,,\frac {2t}k \Big)\cup
                              \cdots\cup
          \Big( \frac{(k-1)tx}k\,,\frac {(k-1)t}k \Big)\cup
          \Big( tx\,,1 \Big). $$
A routine verification shows that $\mes A = \alp$; on the other hand,
\begin{multline*}
    G(A) = G\Big(\Big(\frac{tx}k\,,\frac tk\Big)\Big)
                                                  = \frac tk\,G((x,1)) \\
    = \frac tk\,x\lfl\frac1{1-x}\rfl = \frac {k+1}k\,t
                                = 2\frac{(k+1)(k+2)}{k(k+3)}\,(1-\alp) \\
    = 2\Big( 1+\frac2{k(k+3)}\Big) (1-\alp) > (1-\alp)\lfl 1/\alp\rfl.
\end{multline*}
\end{example}

We are now ready to state the main result of our paper.
\begin{maintheorem}
Let $A\seq(0,1)$ be an open set of measure $\alp\in(0,1]$. Then
  $$ G(A) \le \begin{cases}
                (1-\alp)\lfl 1/\alp\rfl
                   \quad&\text{if \, $0<\alp\le 0.1$}, \\
                (1-\alp+\alp\{1/\alp\})\lfl 1/\alp\rfl
                   \quad&\text{if \, $0.1 \le \alp \le 0.5$}, \\
                2(1-\alp)
                   \quad&\text{if \, $0.5 \le \alp \le 1$}. \\
              \end{cases} $$
\end{maintheorem}

As Examples \refx{first} and \refx{second} show, the first and the last of
these three estimates are sharp while the second estimate can be improved by
at most $\alp\lfr1/\alp\rfr\lfl1/\alp\rfl\le\lfr1/\alp\rfr$. On the other
hand, we suspect that the first estimate can be extended onto the wider range
$\alp\in(0,1/3]$. Indeed, for no $\alp\in(0,1)$ were we able to construct an
open set $A\seq(0,1)$ of measure $\alp$ with $G(A)$ larger than that provided
by Examples \refx{first}--\ref{x:third}.

We believe that with a slight refinement of our method one can replace $0.1$
in the statement of the Main Theorem by a somewhat larger value; perhaps,
$1/8$ or so. However, to get closer to $1/3$ one needs substantially new
ideas.

The rest of this paper is devoted to the proof of the Main Theorem. In
Section \refs{dense} we first settle the case of dense sets ($\alp>0.5$);
this is a simple part of the argument requiring nothing but an iterated
application of the box principle. Using a theorem by Macbeath on addition
of subsets of the torus group, we then derive the estimate for the range
$\alp\ge 0.1$.

The case $\alp\le0.1$ is much subtler. In Section \refs{aux} we collect a
number of auxiliary results from different sources needed to handle this
case. Two more auxiliary results are established in Section \refs{lemmas}.
Having finished the preparations, we complete our treatment of the ``sparse
case'' in Section \refs{sparse}.

\section{Proof of the Main Theorem: $0.1\le\alp\le 1$}\label{s:dense}

We use the box principle in the following ``continuous'' form.
\begin{lemma}\label{l:boxing}
Let $v$ be a positive real number and suppose that an open subset
$A\seq(0,v)$ satisfies $\mes A>v/2$. Then $v\in 2A$.
\end{lemma}

\begin{proof}
Just notice that the sets $A\seq(0,v)$ and
 $v-A:=\{v-a\colon a\in A\}\seq(0,v)$ have a non-empty intersection in view
of $\mes A+\mes(v-A) = 2 \mes A > v$.
\end{proof}

\begin{proof}[Proof of the Main Theorem for $0.5<\alp\le1$]
Suppose that $A\seq(0,1)$ is an open set with $\alp:=\mes A\in(0.5,1]$; we
are to show that $(2(1-\alp),\infty)\seq\cS(A)$.

First we observe that $(2(1-\alp),1]\seq\cS(A)$: for if $2(1-\alp)<u\le 1$
then $\mes (A\cap(0,u)) \ge \alp-(1-u)>u/2$, hence $u\in2A\seq\cS(A)$ by
Lemma \refl{boxing}.

Next we claim that $[1,2\alp)\seq\cS(A)$: this is because for any
$u\in[1,2\alp)$ we can consider $A$ as an open subset of $(0,u)$ of measure
$\alp>u/2$, and the claim follows by Lemma \refl{boxing}.

Finally, we use induction to prove that
$[1,2^k(2\alp-1)+2(1-\alp))\seq\cS(A)$ for any integer $k\ge 1$. For $k=1$
this reduces to $[1,2\alp)\seq\cS(A)$, proven above. Suppose that $k\ge 2$.
Then by the induction hypothesis for any $u$ satisfying
  $$ 2^{k-1}(2\alp-1)+2(1-\alp) \le u < 2^k(2\alp-1)+2(1-\alp) $$
we have
\begin{multline*}
   \mes (\cS(A)\cap(0,u)) \ge \alp + (2^{k-1}(2\alp-1)+2(1-\alp) - 1) \\
          = 2^{k-1}(2\alp-1)+(1-\alp) > u/2,
\end{multline*}
whence $u\in 2\cS(A)=\cS(A)$ by Lemma \refl{boxing}. This proves the Main
Theorem for the case $0.5<\alp\le 1$.
\end{proof}

We recall a classical theorem by Macbeath on addition of torus subsets.\!\!
\footnote{In fact, the original Macbeath' theorem is slightly stronger than
  its version presented below.}
In this theorem and throughout the rest of the paper, for an integer $h\ge 1$
and subsets $A_1\longc A_h$ of an abelian group, we write
  $$ A_1\longp A_h := \{a_1\longp a_h\colon a_1\in A_1\longc a_h\in A_h\}. $$
We use the abbreviation $hA$ for the sum of $h$ instances of the same set
$A$. (For the special case where $A$ is an open subset of $(0,1)$ this
notation was introduced at the beginning of Section \refs{prelim}.)

\begin{theorem}[Macbeath \refb{mc}]\label{t:macbeath}
Suppose that $A$ and $B$ are open subsets of the torus group $\R/\Z$. Then
  $$ \mes(A+B) \ge \min \{ \mes A+\mes B,\, 1 \}. $$
\end{theorem}

\begin{proof}[Proof of the Main Theorem for $0.1\le\alp\le 0.5$]
We actually show that if $A\seq(0,1)$ is a non-empty open set with
$\alp:=\mes A\le0.5$, then
$G(A)\le(1-\alp+\alp\lfr1/\alp\rfr)\lfl1/\alp\rfl$; this strengthening of the
second inequality of the Main Theorem will be used in Section \refs{sparse}
in the course of the proof of the first inequality.

Write $\sig:=\sup A$ and let $A':=\{ \sig^{-1}a \colon a\in A\}$. Consider an
open subset $A_0\seq A'$ such that $\mes A_0=\alp$ and $\sup A_0=1$. Then
$G(A)=\sig G(A')\le\sig G(A_0)\le G(A_0)$, which shows that $\sup A=1$ can be
assumed without loss of generality.

For positive integer $j$, set $S_j:=\cS(A)\cap(j-1,j)$, and denote by $\oA$
and $\oS_j$ the canonical images of $A$ and $S_j$, respectively, in the torus
group $\R/\Z$. In view of the assumption $\sup A=1$ and since $S_j$ are open
sets, we have $S_j+1\seq S_{j+1}$, implying $\oS_j\seq\oS_{j+1}$. Now from
$S_j+A\seq S_j\cup S_{j+1}\cup \{j\}$ we get
$\oS_j+\oA\seq\oS_j\cup\oS_{j+1}\cup\{0\}=\oS_{j+1}\cup\{0\}$, whence
 $\mes S_{j+1}\ge \min\{\mes S_j+\alp,\,1\}$ by Theorem \reft{macbeath}.
Induction by $j$ yields
\begin{equation}\label{e:mesSj}
  \mes S_j \ge \min \{j\alp,\,1\};\quad j\ge 1.
\end{equation}

Let $k:=\lfl 1/\alp\rfl$ and write $T:=\cS(A)\cap(0,k)$. By \refe{mesSj} we
have
  $$ \mes T = \mes S_1 + \mes S_2 \longp \mes S_k
                                            \ge  (k(k+1)/2)\,\alp > k/2. $$
Consequently, applying the third inequality of the Main Theorem (established
above) to the open set $T':=\{k^{-1}t\colon t\in T\}\seq(0,1)$ of measure
 $\mes T'=k^{-1}\mes T>0.5$, we obtain
\begin{multline*}
  G(A) = G(T) = kG(T') \le 2k(1-\mes T') = 2k(1-k^{-1}\mes T) \\
        \le 2k(1-((k+1)/2)\alp)
          = \big(2-(\lfl 1/\alp\rfl+1)\,\alp\big) \lfl 1/\alp\rfl \\
          = \big (1-\alp+\alp\{1/\alp\} \big) \lfl 1/\alp \rfl,
\end{multline*}
as required.
\end{proof}

\section{Auxiliary results}\label{s:aux}

We gather here several results that will be used in Sections \refs{lemmas}
and \refs{sparse} to prove the first inequality of the Main Theorem.

We start with a basic theorem by Cauchy and Davenport. For a positive integer
$p$ by $\Z/p\Z$ we denote the group of residues modulo $p$.
\begin{theorem}[Cauchy-Davenport \cite{b:c,b:da}]\label{t:CD}
For any prime $p$ and any non-empty subsets $A,B\seq\Z/p\Z$ we have
  $$ |A+B| \ge \min\{|A|+|B|-1,p\}. $$
\end{theorem}

Straightforward induction yields
 $|A_1\longp A_h| \ge \min \{|A_1|\longp|A_h|-(h-1),\,p\}$ for any integer
$h\ge 2$ and non-empty subsets $A_1\longc A_h\seq\Z/p\Z$. This readily
implies the following corollary.
\begin{corollary}\label{c:CD}
Let $h$ be a positive integer and $p$ a prime number. Suppose that non-empty
subsets $A_1\longc A_h\seq\Z/p\Z$ satisfy $A_1\longp A_h\neq\Z/p\Z$. Then
  $$ |A_1|\longp |A_h| \le p+h-2. $$
\end{corollary}

The next lemma is due to Freiman. For a set $S\seq\R$ denote by $S\mmod 1$
the canonical image of $S$ in the torus group $\R/\Z$.
\begin{lemma}[Freiman \refb{f}]\label{l:unif}
Let $n$ be a positive integer and let $z_1\longc z_n\in\R/\Z$. Write
$S:=\sum_{j=1}^n e^{2\pi i z_j}$. Then there exists a real number $\bet$ such
that
  $$ \# \{j\in[1,n]\colon z_j\in[\bet,\bet+1/2) \mmod 1 \}
                                                   \ge \frac {n+|S|}2. $$
\end{lemma}

We need yet another well-known result by Freiman.
\begin{theorem}[Freiman, see \protect{\cite[Theorem 1.9]{b:f1}}]\label{t:3n-3}
Let $A$ be a finite set of integers such that $\min A=0$ and $\gcd A=1$.
Write $n:=|A|$ and $l:=\max A$. Then
   $$ |2A| \ge \min \{ l, 2n-3 \} + n. $$
\end{theorem}

Now we quote a theorem by the present author which extends Theorem
\reft{3n-3} to the case of $h$-fold set addition for any positive integer
$h$.
\begin{theorem}[Lev \protect{\cite[Corollary 1]{b:l3}}]\label{t:hA}
Let $A$ be a finite set of integers such that $\min A=0$ and $\gcd A=1$.
Write $n:=|A|,\ l:=\max A$, and $\kap:=\lfl (l-1)/(n-2)\rfl$. Then for any
integer $h\ge 1$ we have
   $$ |hA| \ge \begin{cases}
                  \frac{h(h+1)}2\,(n-2)+h+1\ &\text{if}\ \, h\le\kap, \\
                  \frac{\kap(\kap+1)}2\,(n-2)+\kap+1+(h-\kap)l
                                           \ &\text{if}\ \, h\ge\kap.
               \end{cases} $$
\end{theorem}

\begin{corollary}\label{c:hA}
Let $A$ be a finite set of $n:=|A|\ge 3$ integers such that $\min A=0$ and
 $\gcd A=1$. Suppose that $l:=\max A\le 2n-4$. Then for any
integer $h\ge 1$ we have
   $$ |hA| \ge n+(h-1)l. $$
\end{corollary}

Finally, we present a result which describes the \emph{structure} of the sets
$hA$ and shows that under certain conditions, these sets contain long blocks
of consecutive integers.
\begin{theorem}[Lev, reformulation of {\cite[Theorem 1]{b:l2}}]\label{t:lint}
Let $A$ be a finite set of integers such that $\min A=0$ and $\gcd A=1$.
Write $n:=|A|,\ l:=\max A$, and $\kap:=\lfl(l-1)/(n-2)\rfl$. Then for any
integer $h\ge 2\kap$ we have
  $$ [(2l-(\kap+1)(n-2)-2)\kap, hl-(2l-(\kap+1)(n-2)-2)\kap] \seq hA. $$
\end{theorem}

\section{More auxiliary results}\label{s:lemmas}

In this section we establish two more results needed to complete the proof of
the Main Theorem. The first one is a ``continuous version'' of Theorem
\reft{lint}.
\begin{theorem}\label{t:clint}
Let $A\seq(0,1)$ be an open set such that $\inf A=0$ and $\sup A=1$. Write
$\alp:=\mes A$ and $\kap:=\lfl1/\alp\rfl$. Then for any integer $h\ge2\kap$
we have
  $$ ((2-(\kap+1)\alp)\kap,h-(2-(\kap+1)\alp)\kap) \seq hA. $$
\end{theorem}

A simple rescaling yields the following corollary.
\begin{corollary}\label{c:clint}
Let $A$ be a non-empty bounded open set of real numbers. Write
  $$ v:=\inf A,\ w:=\sup A,\ \lam:=w-v,\ \alp:=\mes A,
                                  \ \kap:=\lfl\lam/\alp\rfl. $$
Then for any integer $h\ge 2\kap$ we have
  $$ (vh+(2\lam-(\kap+1)\alp)\kap, wh-(2\lam-(\kap+1)\alp)\kap) \seq hA. $$
\end{corollary}

\begin{proof}[Proof of Theorem \reft{clint}]
First we prove the assertion assuming that $A$ is a union of a \emph{finite}
number, say $m$, of open intervals. For a sufficiently large integer $l$ (it
will be seen shortly what is sufficient for our purposes) define
  $$ A_l := \{0\} \cup \{z\in\Z\colon [z/l,(z+1)/l)\seq A\} $$
and set $n:=|A_l|$, so that $A_l\seq[0,l-1]$ and
\begin{equation}\label{e:nl}
  \alp l-2m+1 \le n < \alp l+1.
\end{equation}
We have $\min A_l=0,\ \max A_l=l-1$ since $\sup A=1$, and $\gcd A_l=1$ since
$A_l$ contains consecutive integers.

Let $k:=\lfl (l-2)/(n-2) \rfl$. By \refe{nl},
  $$ k \le \frac{l-2}{n-2}<\frac1\alp\,\frac{n+2m}{n-2}
                                                    = \frac1\alp+O(1/n) $$
with the implicit constant depending on $m$ and $\alp$, and it follows that
$k\le \kap$. On the other hand, we claim that $k\ge\kap$. Indeed, this is
clear if $\alp>1/2$ (where $\kap=1$), while for $\alp\le1/2$ from \refe{nl}
we get $n<l/2+1$ whence $l>2n-2$ and therefore
  $$ \frac{l-2}{n-2}>\frac l{n-1} > \frac 1\alp \ge \kap. $$
Thus, $k=\kap$ and by Theorem \reft{lint} as applied to the set $A_l$ we have
  $$ [(2l-(\kap+1)(n-2)-4)\kap,h(l-1)-(2l-(\kap+1)(n-2)-4)\kap] \seq hA_l, $$
hence
  $$ \left[\left(2-\frac{(\kap+1)(n-2)+4}l\right)\kap,
      h-\frac1l-\left(2-\frac{(\kap+1)(n-2)+4}l\right)\kap\right] \seq hA. $$
It remains to notice that
  $$ (\kap+1)(n-2)+4 > (\kap+1)(\alp l-2m-1) = (\kap+1)\alp l + O(1). $$

Now that Theorem \reft{clint}, and therefore Corollary \refc{clint}, are
established when $A$ is a finite union of intervals, we turn to the general
case. Evidently, for any set $A$ satisfying the assumptions of the theorem
and any $\eps>0$ we can find a subset $A'\seq A$ which is a union of a finite
number of intervals and such that $v':=\inf A'<\eps,\ w':=\sup A'>1-\eps$,
and $\alp':=\mes A'>\alp-\eps$. If $\eps$ is small enough then $\lfl
1/\alp'\rfl=\lfl1/\alp\rfl$ and applying Corollary \refc{clint} to the set
$A'$ we get
  $$ (v'h+(2(w'-v')-(\kap+1)\alp')\kap, w'h-(2(w'-v')-(\kap+1)\alp')\kap)
                                                        \seq hA' \seq hA. $$
The assertion follows.
\end{proof}

Another result of this section can be considered as a strengthening of
Corollary \refc{CD} for the special case of equal set summands. Its proof
roughly follows the lines of a well-known argument by Freiman (see
\cite[Theorem~2.1]{b:f1}).

For a positive integer $p$ and an integer set $S$, by $S\mmod p$ we denote
the canonical image of $S$ in the residue group $\Z/p\Z$. The set of non-zero
elements of $\Z/p\Z$ is denoted $(\Z/p\Z)^\times$.
\begin{lemma}\label{l:kp}
For any integer $k\ge 8$ there exists an integer $p_0$ with the following
property. Let $p>p_0$ be a prime number and suppose that a set $A\seq\Z/p\Z$
satisfies $n:=|A|>p/(k+1)$ and $kA\neq\Z/p\Z$. Then $A$ is contained in an
arithmetic progression of at most $(p-2n)/(k-2)+1$ terms.
\end{lemma}

\begin{proof}
We split the argument into several steps.

\mypar For $z\in\Z/p\Z$ set $\hA(z):=\sum_{a\in A} e^{-2\pi i az/p}$ and let
$M:=\max\{|\hA(z)|\colon z\in(\Z/p\Z)^\times\}$. Fix an element $g\in\Z/p\Z$
such that $g\notin kA$. Then
  $$ \sum_{z=0}^{p-1} \big(\hA(z)\big)^k e^{2\pi i\frac{gz}p} = 0 $$
and using Parseval's identity we get
  $$ n^k = - \sum_{z=1}^{p-1} \big(\hA(z)\big)^k e^{2\pi i\frac{gz}p}
               \le \sum_{z=1}^{p-1} \big|\hA(z)\big|^k
      \le M^{k-2} \sum_{z=1}^{p-1} \big|\hA(z)\big|^2 = M^{k-2} n(p-n). $$
It follows that
  $$ M \ge n \left( \frac{n}{p-n}\right)^{1/(k-2)} > \kap n $$
with $\kap:=k^{-1/(k-2)}$. Choose $z\in(\Z/p\Z)^\times$ so that $M=|\hA(z)|$
and therefore $|\hA(z)|>\kap n$. By Lemma \refl{unif} there exists an integer
$u$ such that
  $$ \# \{ a\in A \colon az \in [u,u+(p-1)/2] \mmod p \}
                                                   > \frac{1+\kap}2\,n. $$
Thus, if $d$ is the inverse of $z$ modulo $p$ and if we set $v:=du$ then,
letting $A_1:=\{v,v+d\longc v+((p-1)/2)d\}\cap A$ and $n_1:=|A_1|$, we have
\begin{equation}\label{e:n1}
  n_1>\frac{1+\kap}2\,n.
\end{equation}
Set $B_1:=\{b\in[0,(p-1)/2]\colon v+bd\in A_1\}$. Applying a suitable linear
transformation to $A$, we can assume without violating \refe{n1} that
 $v=0\in A_1$ (whence $\min B_1=0$), $d=1$, and $\gcd(B_1)=1$, and we write
then $l_1:=\max B_1$. Thus $B_1\seq[0,l_1]$ and $A_1=B_1\mmod p$. Evidently,
we have $|B_1|=|A_1|$ and since $l_1\le(p-1)/2$ it is easily seen that
$|2A_1|=|2B_1|$.

\mypar We claim that $l_1<p/6$. First, we prove that $l_1\le 2n_1-4$. For,
assuming $l_1\ge 2n_1-3$, by Theorem \reft{3n-3} we obtain
 $|2A_1|=|2B_1|\ge 3n_1-3$; then in view of
$2A_1+2A_1+(k-4)A\seq kA\neq\Z/p\Z$, applying Corollary \refc{CD} to two
instances of the set $2A_1$ and $k-4$ instances of the set $A$ we get
\begin{multline}\label{e:loc1}
  p+k-4 \ge 2|2A_1|+(k-4)|A| > (6n_1-6)+(k-4)\,n \\
          > (3(\kap+1)+(k-4))\,n - 6 > \frac{k-1+3\kap}{k+1}\,p-6  \\
              =  p + \frac{3\kap-2}{k+1}\,p - 6.
\end{multline}
However, since $3\kap>2$ this is wrong for sufficiently large $p$, a
contradiction.

We see that $l_1\le 2n_1-4$ and applying Theorem \reft{3n-3} once again we
get
\begin{equation}\label{e:2B1}
  |2A_1|=|2B_1|\ge n_1+l_1.
\end{equation}
If $k$ is even then from
 $(k/2)(2A_1)\seq kA\neq\Z/p\Z$, applying Corollary \refc{CD} to $k/2$
instances of the set $2A_1$ we get
  $$ p + \frac k2 - 2 \ge \frac k2\,|2A_1| \ge \frac k2\,(n_1+l_1), $$
hence
  $$ l_1 < \frac 2k\,p-n_1+1
        < \left(\frac 2k - \frac{\kap+1}{2(k+1)}\right)\,p+1 < \frac 16\,p, $$
as the expression in the parentheses is smaller than $1/6$ for $k\ge 8$.
Similarly, if $k$ is odd then from $\frac{k-1}2\,(2A_1)+A\seq kA\neq\Z/p\Z$
we obtain
  $$ p + \frac{k+1}2 -2 \ge \frac{k-1}2\,(l_1+n_1)+n $$
which yields
\begin{multline*}
  l_1 < \frac2{k-1}\,p - \left(\frac2{k-1}+\frac{\kap+1}2\right)n + 1 \\
          < \left( \frac2{k-1} - \frac2{k^2-1}
                 - \frac{\kap+1}{2(k+1)} \right) p + 1 \\
             = \left( \frac{2k}{k^2-1} - \frac{\kap+1}{2(k+1)} \right) p + 1
       < \frac16\,p,
\end{multline*}
as the expression in the last pair of parentheses is smaller than $1/6$ for
$k\ge 9$.

\mypar Our next claim is that $A\seq[-l_1,2l_1]\mmod p$; informally, if we
extend the interval $[0,l_1]\mmod p$ (in which $A_1$ resides) by $l_1$ in
both directions, then the resulting interval covers the whole set $A$. This
follows from the observation that if $a\in A$ is \emph{not} contained in this
extended interval, then $2A_1$ and $a+A_1$ are disjoint subsets of $2A$;
hence $|2A|\ge|a+A_1|+|2A_1|\ge 2n_1+l_1\ge 3n_1-1$ by \refe{2B1}; this,
however, leads to a contradiction as in \refe{loc1}: just notice that
$2A+2A+(k-4)A\neq\Z/p\Z$ and apply Corollary \refc{CD} to two instances of
$2A$ and $k-4$ instances of $A$. This shows that $A$ is contained in a set of
$3l_1+1\le (p+1)/2$ consecutive elements of $\Z/p\Z$, and applying again a
suitable linear transformation to $A$ we can assume that $A=B\mmod p$ for a
set $B\seq[0,(p-1)/2]$ such that $0\in B$ and $\gcd B=1$.

\mypar Let $l:=\max B$. We have
\begin{equation}\label{e:l2n-4}
  l\le 2n-4:
\end{equation}
for otherwise $|2A|=|2B|\ge 3n-3$ by Theorem \reft{3n-3} and again we get a
contradiction as in \refe{loc1}. Applying Theorem \reft{3n-3} once again we
get $|2A|=|2B|\ge l+n$. Now by Corollary \refc{CD} for $k$ even from
$(k/2)(2A)\neq\Z/p\Z$ we obtain
\begin{gather}
  p + \frac k2 - 2 \ge \frac k2\, |2A| \ge \frac k2\, (l+n), \notag \\
  l < \frac2k\,p-n+1 < \left( \frac 2k-\frac1{k+1}\right)p + 1
                                             < \frac p{k-1} \label{e:loc2}
\end{gather}
and for $k$ odd from $\frac{k-1}2(2A)+A\neq\Z/p\Z$ we obtain
\begin{gather}
  p +\frac {k+1}2 - 2 \ge \frac{k-1}2\,|2A| + |A|
                               \ge \frac {k-1}2\, (l+n) + n, \notag \\
  l < \frac2{k-1}\,p-\left(1+\frac 2{k-1}\right)n + 1
        = \frac2{k-1}\,p-\frac {k+1}{k-1}\,n + 1
                                          < \frac p{k-1}+1  \label{e:loc3}.
\end{gather}

\mypar From \refe{loc2}, \refe{loc3}, and the definition of $l$ we conclude
that
\begin{equation}\label{e:k-1A}
  |(k-1)A|\ge|(k-1)B|-k,
\end{equation}
hence recalling \refe{l2n-4} and using Corollary \refc{hA} we get $|(k-1)A|
\ge n+(k-2)l-k$. Now applying Corollary \refc{CD} to the sets $(k-1)A$ and
$A$ we obtain
\begin{equation}\label{e:nklp}
  p\ge 2n+(k-2)l-k,
\end{equation}
implying
  $$ l < \frac{p-2n}{k-2}+2 < \frac{1-2/(k+1)}{k-2}\,p + 2 < \frac p{k-1} $$
an thus providing a small, but important strengthening of \refe{loc3}. Now
\refe{k-1A} can be strengthened to $|(k-1)A|=|(k-1)B|$, and accordingly
\refe{nklp} to $p\ge 2n+(k-2)l$. Therefore $l\le(p-2n)/(k-2)$ from which the
assertion follows. This proves Lemma \refl{kp}.
\end{proof}

\section{Proof of the Main Theorem: $0<\alp\le 0.1$}\label{s:sparse}

\begin{proof}[Proof of the Main Theorem for $0<\alp\le 0.1$]
Write $k:=\lfl1/\alp\rfl\ge 10$ and suppose, for a contradiction, that
$G(A)>(1-\alp)k$. As in the proof of Lemma \refl{kp}, we proceed by steps.

\setcounter{mypar}{0}

\mypar Without loss of generality we can assume that $A$ satisfies the
following three conditions:
\begin{itemize}
\item[(i)]  $A$ is a \emph{finite} union of open intervals;
\item[(ii)] $(2A)\cap (0,1)=A$ (informally, ``$A$ is closed under addition
  in $(0,1)$'');
\item[(iii)] $\sup A=1$.
\end{itemize}
For if (i) fails, we consider a subset $A'\seq A$ which is a finite union of
open intervals and the measure of which $\alp':=\mes A'$ is sufficiently
close to $\alp$; specifically, we request that $\lfl1/\alp'\rfl=k$ and
$G(A)>(1-\alp')k$. Then $G(A')\ge G(A)>(1-\alp')k$ and we simply replace $A$
with $A'$ and $\alp$ with $\alp'$.

Now assume that (i) holds and replace $A$ with $A'':=\cS(A)\cap(t,1)$, where
$t\in[0,1)$ is so chosen that $\mes A''=\alp$. Clearly, this new set remains
a finite union of open intervals and in addition satisfies
$(2A'')\cap(0,1)=A''$. Also, $G(A'')\ge G(A)>(1-\alp)k$.

Eventually, assuming both (i) and (ii) write $\sig:=\sup A$ and replace $A$
with $A''':=\{\sig^{-1}a\colon a\in A\}\cap(t,1)$, where (as above) we choose
$t\in[0,1)$ so that $\mes A'''=\alp$. Then (i) and (ii) remain valid and
(iii) becomes true, too, while $G(A''')\ge G(A)>(1-\alp)k$.

Let $m$ denote the number of intervals in $A$.

\mypar Write $\oA:=A\mmod 1$; we claim that
\begin{equation}\label{e:kAnotall}
  k(\oA\cup\{0\})\neq\R/\Z.
\end{equation}
Assume the opposite. Then for any real $u\in(k-1,k)$ there exist integers
$s\in[1,k]$ and $z\ge 0$ and elements $a_1\longc a_s\in A$ such that
$u=a_1\longp a_s+z$. Since $A$ is an open set and in view of $\sup A=1$ we
conclude that $u\in \cS(A)$, whence
\begin{equation}\label{e:k-1kcS}
  (k-1,k)\seq\cS(A).
\end{equation}
However, as shown in Section \refs{dense} (see remark at the beginning of
the proof of the Main Theorem for $0.1\le\alp\le 0.5$) we have
\begin{equation}\label{e:1-alp}
  G(A)\le (1-\alp+\alp\{1/\alp\})k < k.
\end{equation}
Comparing \refe{k-1kcS} and \refe{1-alp} we get $G(A)\le k-1\le(1-\alp)k$, a
contradiction establishing \refe{kAnotall}.

Fix a large prime number $p$ (it can be figured out from the subsequent
argument exactly how large $p$ is to be) and define
  $$ A_p := \{0\}\cup\{z\in\Z\colon [z/p,(z+1)/p)\seq A\}, $$
so that $A_p\seq[0,p-1]$. Write $\oA_p:=A_p\mmod p$ and let $z\in[1,p-1]$. If
the canonical image $\oz$ of $z$ in $\Z/p\Z$ satisfies $\oz\in k\oA_p$, then
clearly $[z/p,(z+1)/p)\mmod 1\seq k(\oA\cup\{0\})$. It follows that
$k\oA_p=\Z/p\Z$ would imply $[1/p,1)\mmod 1\seq k(\oA\cup\{0\})$, which is
wrong for sufficiently large $p$ in view of \refe{kAnotall} and $0\in
k(\oA\cup\{0\})$. Thus
\begin{equation}\label{e:koA_p}
  k\oA_p\neq\Z/p\Z.
\end{equation}

\mypar Set $n:=|\oA_p|$. By (i), for $p$ large enough we have
\begin{equation}\label{e:nalpp}
  n=\alp p+O(1)>p/(k+1)
\end{equation}
(with the implicit constant depending on the number of intervals $m$). By
\refe{koA_p}, \refe{nalpp}, and Lemma \refl{kp}, there exist
 $\od\in [1,(p-1)/2]\mmod p$ and non-negative integers $n_1,n_2$ such that
\begin{equation}\label{e:n1n2large}
  n_1+n_2\le \frac{p-2n}{k-2}
\end{equation}
and $\oA_p\seq\{-n_1\od\longc -\od,0,\od\longc n_2\od\}$. (Recall, that
 $0\in A_p$.) Clearly, we can assume without loss of generality that
$-n_1\od,n_2\od\in\oA_p$. Let $d\in[1,(p-1)/2]$ be the integer such that
$\od$ is the canonical image of $d$.

Evidently, $\oA_p$ is a union of at most $m+1$ arithmetic progressions with
difference $1$. The longest of these progressions has at least $n/(m+1)$
terms, hence there exist integers $s\ge n/(m+1)$ and
 $x_1\longc x_s\in[-n_1,n_2]$ such that
\begin{equation}\label{e:chain}
  (x_{i+1}-x_i)d\equiv 1\mmod p;\qquad i=1\longc s-1.
\end{equation}
Define $d'\in(-p/2,p/2)$ by $dd'\equiv 1\mmod p$. Multiplying \refe{chain} by
$d'$ we get $x_{i+1}-x_i\equiv d'\mmod p$ whence $x_{i+1}-x_i=d'$ and
therefore $x_s-x_1=(s-1)d'$, implying
\begin{equation}\label{e:d5m}
  |d'| \le \frac {n_2+n_1}{s-1} < \frac{p}{k-2}\,\frac{3m}n
       < \frac{k+1}{k-2}\,\cdot 3m < 5m
\end{equation}
by \refe{nalpp} and \refe{n1n2large}.

We have either $|dd'|=1$, in which case $d=1$, or $|dd'|\ge p-1$, and then
\begin{equation}\label{e:dlarge}
  d \ge \frac{p-1}{5m-1} > \frac p{5m}
\end{equation}
by \refe{d5m}. We proceed to show that indeed $d=1$.

\mypar Assume that \refe{dlarge} holds true. We observe that property (ii)
implies
\begin{equation}\label{e:half}
  2(\oA_p \cap [0,(p-1)/2]\mmod p)\seq \oA_p.
\end{equation}
It follows, in particular, that
\begin{equation}\label{e:n1d}
  n_1\od\in[0,(p-1)/2]\mmod p:
\end{equation}
for otherwise $-n_1\od\in[1,(p-1)/2]\mmod p$ and then $-2n_1\od\in\oA_p$ by
\refe{half}, hence $-2n_1\in[-n_1,n_2]$, which is wrong. Similarly,
\begin{equation}\label{e:n2d}
  n_2\od\in[(p+1)/2,p]\mmod p.
\end{equation}

Set
\begin{align*}
  N_1 &:= \# \{ x\in[-n_1,-n_1/2)\colon x\od\in[0,(p-1)/2]\mmod p \} \\
      &\;=  \# \{ x\in(n_1/2,n_1]\colon x\od\in[(p+1)/2,p]\mmod p \} \\
\intertext{and}
  N_2 &:=  \# \{ x\in(n_2/2,n_2]\colon x\od\in[0,(p-1)/2]\mmod p \}.
\end{align*}
We show that
\begin{equation}\label{e:N1N2}
  N_i \ge \frac{n_i}8 - \frac p{2d};\qquad i=1,2;
\end{equation}
then in view of \refe{dlarge}, \refe{n1n2large}, and \refe{nalpp} we will get
\begin{align*}
  N_1+N_2 &\ge \frac{n_1+n_2}8-\frac pd \\
          &>   n_1 + n_2 - \frac 78\,(n_1+n_2) - 5m \\
          &\ge n_1 + n_2 - \frac 78\, \frac{p-2n}{k-2} - 5m \\
          &=   n_1 + n_2 - n + \frac{4k-1}{4(k-2)}\,n
                                    - \frac 78\,\frac{p}{k-2} - 5m \\
          &>   n_1 + n_2 - n + \frac 1{8(k-2)}
                    \,\left( \frac{2(4k-1)}{k+1} - 7 \right)p - 5m \\
          &=   n_1 + n_2 - n + \frac {k-9}{8(k-2)(k+1)}\,p - 5m \\
          &>   n_1 + n_2 + 1 - n
\end{align*}
so that there exists an integer $x\in[-n_1,-n_1/2)\cup(n_2/2,n_2]$ such that
$x\od\in\oA_p\cap[0,(p-1)/2]\mmod p$. Considering the doubling $2x\od$ and
taking into account \refe{half} we obtain then a contradiction.

\mypar As we have just shown, to prove that $d=1$ it suffices to establish
\refe{N1N2}. We address the estimate of $N_1$ only; $N_2$ can be estimated in
a similar way. Clearly, we can assume that $n_1>0$, and then by \refe{n1d}
there exists a non-negative integer $t$ such that
  $$ \frac {2tp}{2d} \le n_1 < \frac{(2t+1)p}{2d}. $$
For any integer $s\in[(t/2)+1,t]$ we have
  $$ \frac{(2s-1)p}{2d} \ge \frac{(t+1)p}{2d} > \frac{n_1}2
        \quad\text{and}\quad \frac{2sp}{2d} \le \frac{2tp}{2d} \le n_1, $$
hence the interval
  $$ I_s := \left( \frac{(2s-1)p}{2d},\, \frac{2sp}{2d} \right] $$
satisfies $I_s\seq(n_1/2,n_1]$. The length of $I_s$ is $p/(2d)>1$; thus the
number of integers $x\in I_s$ is at least $p/(4d)$ and evidently,
 $\od x\in[(p+1)/2,p]\mmod p$ for any such integer. It follows that
  $$ N_1 \ge \sum_{(t/2)+1\le s\le t} |I_s|
      \ge \lfl \frac t2 \rfl \frac p{4d}
      > \left( \frac t2-1 \right) \frac p{4d}
      > \frac{n_1}8-\frac{p}{2d}, $$
as required.

\mypar We proved that $d=1$ and comparing \refe{n2d} with \refe{n1n2large} we
conclude that $n_2=0$; that is, $\oA_p\seq[-n_1,0]\mmod p$. Letting
$\tau:=\inf A$ and $\lam:=1-\tau$ we obtain
  $$ \tau \ge 1-\frac{n_1+1}p
              \ge 1-\left(\frac{p-2n}{k-2}+1\right) p^{-1}
                                       = 1 - \frac{1-2n/p}{k-2} - \frac1p; $$
recalling \refe{nalpp} we get $\tau\ge 1-(1-2\alp)/(k-2)$, whence
  $$ \lam \le \frac{1-2\alp}{k-2} = \alp\,\left(1+\frac{1/\alp-k}{k-2}\right)
        < \alp\,\left(1+\frac1{k-2} \right) < 2\alp. $$

To complete the proof we invoke Corollary \refc{clint} which gives
  $$ (h\tau+2(\lam-\alp),h-2(\lam-\alp)) \seq hA \seq \cS(A) $$
for any integer $h\ge 2$. It remains to observe that if $h\ge k$, then
\begin{align*}
  h-2(\lam-\alp)
    &=   (h+1)\tau + (h+1)\lam - 1 - 2(\lam-\alp) \\
    &=   (h+1)\tau + 2(\lam-\alp) + (4\alp + (h-3)\lam - 1) \\
    &\ge (h+1)\tau + 2(\lam-\alp) + (h+1)\alp - 1 \\
    &>   (h+1)\tau + 2(\lam-\alp)
\end{align*}
and therefore
  $$ G(A) \le k\tau+2(\lam-\alp) = k-(k-2)\lam-2\alp
                                     \le k-(k-2)\alp-2\alp = (1-\alp)k, $$
contrary to the assumption. This completes the proof of the Main Theorem.
\end{proof}

\bigskip


\vfill
\begin{thebibliography}{EG72}
\bibitem[B03]{b:b} {\sc D.~Bleichenbacher},
  \emph{The continuous postage stamp problem},
  preprint (2003).
\bibitem[C13]{b:c} {\sc A.L.~Cauchy},
  Recherches sur les nombres,
  \emph{J.~\'Ecole polytech.} {\bf 9} (1813), 99--116.
\bibitem[D35]{b:da} {\sc H.~Davenport},
  On the addition of residue classes,
  \emph{J.~London Math. Soc.} {\bf 10} (1935), 30--32.
\bibitem[D90]{b:d} {\sc J.~Dixmier},
  Proof of a conjecture by Erd\H os and Graham concerning the problem of
  Frobenius,
  \emph{J. Number Theory} {\bf 34} (2) (1990), 198--209.
\bibitem[F62]{b:f} {\sc G.A.~Freiman},
  Inverse problems of additive number theory. VII. The addition of finite
  sets. IV. The method of trigonometric sums, (Russian)
  \emph{Izv. Vys\v s. U\v cebn. Zaved. Matematika} {\bf 6} (31) (1962),
  131--144.
\bibitem[F66]{b:f1} {\sc \bysame},
  \emph{Nachala strukturno\u{\i} teorii slozheniya mnozhestv} (Russian)
  [Elements of a structural theory of set addition],
  Kazan. Gosudarstv. Ped. Inst; Elabu\v z. Gosudarstv. Ped. Inst.,
  Kazan, 1966.
  English translation in: Translations of Mathematical Monographs,
  Vol. 37, American Mathematical Society, Providence, R.~I., 1973.
\bibitem[EG72]{b:eg} {\sc P.~Erd\H os} and {\sc R.L.~Graham},
  On a linear diophantine problem of Frobenius,
  \emph{Acta Arith.} {\bf 21} (1972), 399--408.
\bibitem[LP03]{b:lp} {\sc H.W.~Lenstra, Jr.} and {\sc C.~Pomerance},
  \emph{Primality testing with Gaussian periods},
  preprint (2003).
\bibitem[L96a]{b:l1} {\sc V.~Lev},
  On the extremal aspect of the Frobenius problem,
  \emph{J.~Combin. Theory Ser.~A} {\bf 73} (1) (1996), 111--119.
\bibitem[L96b]{b:l3} {\sc \bysame},
  Structure theorem for multiples addition and the Frobenius problem,
  \emph{J.~Number Theory} {\bf 58} (1) (1996), 79--88.
\bibitem[L97]{b:l2} {\sc \bysame},
  Optimal representations by sumsets and subset sums,
  \emph{J.~Number Theory} {\bf 62} (1) (1997), 127--143.
\bibitem[M53]{b:mc} {\sc A.M.~Macbeath},
  On measure of sum sets. II. The sum-theorem for the torus,
  \emph{Proc. Cambridge Philos. Soc.} {\bf 49} (1953), 40--43.
\end{thebibliography}
\end{document}